\documentclass[9.5pt,reqno]{amsart}
\usepackage{amssymb,mathrsfs,graphicx}
\usepackage{epsfig,subfigure}
\usepackage{ifthen}
\usepackage{float}
\usepackage{colortbl}

\usepackage{ifthen} 
\provideboolean{shownotes} 
\setboolean{shownotes}{true} 
\newcommand{\margnote}[1]{
\ifthenelse{\boolean{shownotes}}%
{\marginpar{\raggedright\tiny\texttt{#1}}}%
{}%
}

\newcommand{\hole}[1]{
\ifthenelse{\boolean{shownotes}}%
{\begin{center} \fbox{ \rule {.25cm}{0cm} \rule[-.1cm]{0cm}{.4cm}
\parbox{.85\textwidth}{\begin{center} \texttt{#1}\end{center}} \rule
{.25cm}{0cm}}\end{center}} {} }

\title[An optimal transport approach of hypocoercivity]{An optimal transport approach of hypocoercivity for the 1d kinetic Fokker-Plank equation}

\author[Salem]{Samir Salem}
\address[Samir Salem]{\newline Ecole Polytechnique, CMLS 91128 Palaiseau Cedex, France}
\email{samir.salem@polytechnique.edu}

\numberwithin{equation}{section}

\newtheorem{theorem}{Theorem}[section]
\newtheorem{lemma}{Lemma}[section]

\newtheorem{proposition}{Proposition}[section]

\newtheorem{definition}{Definition}[section]

\newcommand{\R}{\mathbb R}

\newcommand{\HH}{\mathcal{H}}
\newcommand{\II}{\mathcal{I}}

\newcommand{\MM}{\mathcal{M}}

\newcommand{\XX}{\mathcal{X}}
\newcommand{\YY}{\mathcal{Y}}
\newcommand{\PP}{\mathcal{P}}

\newcommand{\CC}{\mathcal{C}}

\newcommand{\JJ}{\mathcal{J}}

\newcommand{\mt}{\mathcal{T}}
\newcommand{\lal}{\langle}
\newcommand{\ral}{\rangle}

\newcommand{\lt}{\left}
\newcommand{\rt}{\right}
\newcommand{\pa}{\partial}
 
\newcommand{\mb}{\mathbf{1}}

\newcommand{\bq}{\begin{equation}}
\newcommand{\eq}{\end{equation}}

\def\charf {\mbox{{\text 1}\kern-.30em {\text l}}}

\begin{document}
\allowdisplaybreaks

\date{\today}



\begin{abstract} 
A quadratic optimal transport metric on the set of probability measure over $\R^2$ is introduced. The quadratic cost is given by the euclidean norm on $\R^2$ associated to some well chosen symmetric positive matrix, which makes the metric equivalent to the usual Wasserstein-2 metric. The dissipation of the distance to the equilibrium along the kinetic Fokker-Planck flow, is bounded by below in terms of the distance itself. It enables to obtain some new type of trend to equilibrium estimate in Wasserstein-2 like metric, in the case of non-convex confinement potential. 
\end{abstract}

\maketitle \centerline{\date}

%
%
%
%

\section{Introduction}\label{intro}
The kinetic Fokker-Planck equation (in one dimension of space) 
\bq
\label{eq:PS}
\pa_t f_t+v\pa_x f_t- U'(x)\pa_v f_t=\pa_v \lt( vf_t + \pa_v f_t\rt),
\eq
is a kinetic model which describes the time evolution of the law of some particle in the phase space, submitted to some Brownian force, velocity friction and evolving in some confinement potential $U\in \CC^2(\R)$
\bq
\label{eq:LGV}
\begin{aligned}
	&dX_t= V_t dt\\
	&dV_t= -U'(X_t)dt-V_tdt +\sqrt{2}dW_t
\end{aligned}
\eq
where $(W_t)_{t\geq 0}$ is a Brownian motion.\\
The well posedness theory for \eqref{eq:PS} is classical, since all the coefficients are smooth. \newline
As the diffusion operator is degenerated and acts only on velocity, the theories of regularity and trend to equilibrium are more delicate. The latter is known as \textit{hypocoercivity} since Villani's eponymous memoir \cite{VilHypo}.\newline
In order to roughly sketch the core idea of this theory, let us sat that for some $a,b,c>0$ well chosen, a twisted norm on the Sobolev space weighted by the equilibrium  measure
\[
f_\infty(x,v)=Z^{-1}e^{-U(x)-\frac{|v|^2}{2}}, \ Z=\sqrt{2\pi}\int_{\R} e^{-U(y)}dy,
\]
(the confusion between probability measures and densities may be throughout the paper), is defined as
\[
\|\cdot\|^2_{\tilde{H}_1}=\|\cdot\|^2_{L^2(f_\infty)}+a\|\pa_x \cdot\|^2_{L^2(f_\infty)}+b\lal \pa_x \cdot , \pa_v \cdot \ral_{L^2(f_\infty)}+c\|\pa_x \cdot\|^2_{L^2(f_\infty)}.
\]
It is equivalent to the usual norm on $H^1(f_\infty)$, provided that the quadratic form associated to the triplet $(a,b,c)$ is positive definite.\newline
But compared to the usual norm, it includes the cross term $b\lal \pa_x \cdot , \pa_v \cdot \ral$ which, when differentiated along \eqref{eq:PS}, enables to conjugate the effects of the diffusion $\pa_v^2$ in velocity and the free transport $v\pa_x$ in the phase space. The result (see \cite[Theorem 27]{VilHypo}) is that equation \eqref{eq:PS} is a contraction for this specific norm, i.e.
\[
\frac{1}{2}\frac{d}{dt}\|f_t/f_\infty-1\|^2_{\tilde{H}_1}\leq - \kappa \|f_t/f_\infty-1\|^2_{\tilde{H}_1}. 
\]
It is obtained under the assumption that the confinement potential $U$ (up to an additive constant) is such that the probability measure $e^{-U}$ satisfies a Poincaré inequality  \newline
Under the stronger assumption that $e^{-U}$ satisfies a logartihmic Sobolev inequality (see \cite[Section 9.2]{VilOT}), a similar result is also derived, in terms of relative entropy w.r.t. the equilibrium measure. A typical sufficient condition for log-Sobolev inequality to be satisfied, is for instance that $U$ can be decomposed as the sum of some strictly convex and a bounded functions (see \cite[Theorem 9.9 ]{VilOT}). More precisely, the decay estimate
\bq
\label{eq:HVill}
\HH(f_t|f_\infty):=\int_{\R} f_t \ln \lt( \frac{f_t}{f_\infty} \rt)\leq C \frac{e^{-\lambda (t-t_0)}}{t_0^3}\HH(f_0|f_\infty), \ \forall t_0\in (0,1),
\eq
is obtained (see \cite[Remark 41 (7.11)]{VilHypo}. Since $e^{-U}$ satisfies a $\lambda$-log-Sobolev inequality, $f_\infty$ satisfies $(\lambda\wedge 1)$-log-Sobolev inequality, and thus a transport inequality (see \cite[section 9.3.1]{VilOT}), which means that for any $f\in \PP_2(\R^2)$
\[
W_2^2(f,f_\infty)\leq \frac{\HH(f|f_\infty)}{\lambda\wedge 1}.
\]
In particular, it is clear that $W_2(f_t,f_\infty)=O(e^{-\frac{\kappa}{2}t})$, and the time asymptotic behavior of the distance of the solution to \eqref{eq:PS} and the equilibrium measure is well understood. \newline

Under the more restrictive condition that $U$ is strictly convex, Bolley et al. \cite{BGGKin}, extended in the core idea of hypocoercivity to the framework of coupling metric, by introducing a twisted Wasserstein-2 metric, under which equation \eqref{eq:PS} is a contraction. \newline
More precisely for $r,s>0$ well chosen, and $\mu,\nu\in \PP_2(\R^2)$, a distance $d_Q$ is defined as
\[
d_Q^2(\mu,\nu)=\inf_{\gamma\in \Pi(\mu,\nu)}\int_{\R^{2} \times \R^{2}}  (r|x_1-x_2|^2+2(x_1-x_2)(v_1-v_2)+s|v_1-v_2|^2)\gamma(dz_1,dz_2),
\]
where $\Pi(\mu,\nu)$ is the choice of probability measures on $\R^2\times \R^2$ admitting respectively $\mu$ and $\nu$ as marginals. Compared to the usual Wasserstein-2 metric, there is a cross term between position and velocity, similarly as what is done in the $L^2$ theory. \newline
This enables the authors to show that there is $\kappa>0$ such that if $(f_t)_{t\geq 0},(g_t)_{t\geq 0}$ are the solutions to \eqref{eq:PS} starting from $f_0,g_0$ respectively, there holds for any $t\geq 0$
\[
d_Q(f_t,g_t)\leq e^{-\kappa t}d_Q(f_0,g_0).
\]
In particular, since for the set of $r,s$ the quadratic form associated to $(r,2,s)$ is positive definite, $d_Q$ is metrically equivalent to the usual Wasserstein-2 metric, it provides a new way to estimate of the exponential convergence of $W_2(f_t,f_\infty)$. But the estimate is somehow smoother than \eqref{eq:HVill} for short time.\newline
This result is obtained by so-called synchronous coupling, which means that the quantity of interest is the expectation of the squared distance between two solutions to \eqref{eq:LGV} for two distinct initial conditions of law $f_0$ and $g_0$, and the same Brownian motion.\newline
The fact that the Brownian motion is the same for both trajectories, roughly means that the technique is actually indifferent to the presence of diffusion. And that is why the strict convexity assumption on $U$ is needed. In the case with no diffusion, the physical example of a bullet evolving in a gap with friction is instructive to realize that strict convexity of the confinement is needed to ensure uniqueness of stationary state. And since no advantage is taken form the diffusion in velocity, the analogy with the standard $L^2$ theory is limited.\newline 

In this paper, we aim at presenting through a simple 1d toy example, how to extend this idea of twisted Wasserstein metric to obtain a $W_2$ hypocoercivity theory closer to the standard $L^2$ theory (and also to the relative entropy theory by some aspects). For that purpose, we consider some locally non convex potential $U$ satisfying : 
\bq
\label{eq:U}
U(x)=\alpha \frac{|x|^2}{2}+\psi(x)
\eq

In that case, the fact that $W_2(f_t,f_\infty)$ converges to $0$ exponentially fast is well known (as noted above). But the non convexity of the confinement prevents from obtaining contraction estimate in equivalent metric, by using standard coupling techniques. We then extend the idea of \cite{BGG1,BGG2}, which consists in estimating the dissipation of the optimal transport distance itself, and not some quantity which bounds it by above. The argument strongly relies on the optimal transportation map between the solution and the equilibrium measure, and the fact that, in this quadratic framework, this map derives from some convex potential \textit{à la} Brenier.   \newline

Obtaining some exponential stability estimates in optimal transport or coupling distances, for kinetic Fokker-Planck in non convex landscape, is a problem that has dragged a lot of attention from the kinetic community lately. We can mention the recent \cite{GM}, which relies on some entropy method, as roughly described above. In \cite{DET}, the authors study the equation set on the torus in space, with no confinement, and obtain some exponential stability estimate in usual $W_2$ distance, up to some multiplicative constant, with a probabilistic approach. \newline
Finally, let us mention that for the weaker $W_1$ metric, recent works \cite{Eva},\cite{Ebe} based on sophisticated markovian coupling, have proved their efficiency to solve the problem of non convex landscape.

%
%

\section{Preliminaries and main results}

\subsection{Optimal transport metric on $\R\times \R$} We begin with some optimal transport considerations. \newline

The usual Wasserstein-2 metric on $\PP_2(\R\times \R)$ is defined as 
\[
W^2_2(\mu,\nu)=\inf_{\mt=(\mt^1,\mt^2), \mt\#\nu=\mu}\int_{\R\times \R}\lt|\mt^2(x,v)-x\rt|^2+|\mt^2(x,v)-v|^2\nu(dx,dv),
\]
for $\mu$ and $\nu$ admitting smooth enough densities (see for instance \cite[Theorem 2.12]{VilOT}). There exists a unique map $\mt$ which achieves the infimum, and it is given as the gradient of some (up to an additive constant) convex potential.\newline
As noted in the above introduction, in order to mimic the $L^2$ hypocoercivity theory, we need some cross term between position and velocity. In that purpose we introduce some positive definite matrix 
\[
A=\begin{pmatrix}
a  & b\\ 
b  & c
\end{pmatrix} 
\]
with $A^{-1}=\det(A)^{-1}\begin{pmatrix}
c  & -b\\ 
-b  & a
\end{pmatrix} $. We denote the eigenvalues of this matrix
\[
\nu_1=\frac{c+a-\sqrt{(c-a)^2+4b^2}}{2}, \ \nu_2=\frac{c+a+\sqrt{(c-a)^2+4b^2}}{2},
\]
and define the equivalent norm on $\R^2$
\[
|(x,v)|^2_A:=(x,v)\cdot A(x,v)=a|x|^2+2bxv+c|v|^2.
\]
This twisted euclidean norm is the cost function associated to the optimal transport metric given in the 

\begin{definition}
	For $A$ a symmetric positive definite matrix,we define $W_A$ for any $\mu, \nu\in \PP_2(\R^2)$ as 
 \[
 W_A^2(\mu,\nu)=\inf_{\gamma \in \Pi(\mu,\nu)}\int_{\R^2}|(x_1,v_1)-(x_2,v_2)|_A^2\gamma(dx_1,dv_1,dx_2,dv_2),
 \]
 where
 \[
 \Pi(\mu,\nu)=\Bigl\{ \gamma \in \PP(\R^2\times \R^2), \ \gamma(K\times \R^2)=\mu(K), \ \gamma(\R^2\times K)=\nu(K), \ \forall K\subset \R^2 \Bigr\}.
 \]
Moreover, $W_A$ is a distance on $\PP_2(\R^2)$. 
\end{definition}  
See for instance \cite[Theorem 1.3]{VilOT} for the well posedness and \cite[Theorem 7.3]{VilOT} for the metric aspect. The main result of this section, is a result of polarization of the optimal transport map given in the
\begin{theorem}[Brenier's theorem for twisted Wasserstein metric]
	\label{thm:Bre}
	For any $\mu,\nu\in \PP_2(\R^2)$ having $\CC^1$, strictly positive densities :
	\begin{itemize}
	\item[$(i)$] there is (unique up to additive constant) $\varphi$ convex such that $A^{-1}\nabla\varphi\# \nu=\mu$ and
	\[
	W_A^2(\mu,\nu)=\int_{\R^{2}}\lt|(x,v)-A^{-1}\nabla \varphi(x,v)\rt|^2_A\nu(dx,dv),
	\]
	\item[$(ii)$] if $\varphi^*$ denotes the convex conjugate of $\varphi$, then $\nabla \varphi^*(A\cdot)\#\mu = \nu$ and
	\[
	W_A^2(\mu,\nu)=\int_{\R^{2}}\lt|(x,v)-\nabla \varphi^*(A(x,v))\rt|^2_A\mu(dx,dv),
	\]
	\item[$(iii)$] $\varphi$ is $\CC^2(\R^2)$, (in particular $\pa_x^2\varphi,\pa_v^2\varphi>0$) and $\forall(x,v)\in \R^2$ there holds
	\[
	\nabla \varphi^*(\nabla \varphi(x,v))=(x,v).
	\]  
	\end{itemize}
\end{theorem}

The proof of this result can be found in the below Appendix. But it mostly consists in a simple application of \cite[Theorem 1.17]{OTAM} or \cite[Theorem 2.44]{VilOT}, and regularity theory for the Monge-Ampère equation. One of the main feature of this metric is given in the 
\begin{proposition}
	\label{prop:equ}
	The distance $W_A$ is metrically equivalent to $W_2$.
\end{proposition}	
\begin{proof}
	Let $\mu,\nu\in \PP_2(\R^2)$. 
	Let $\mt$ be the optimal transport map from $\nu$ toward $\mu$ w.r.t. $W_2$ (resp. $\mt'$ the optimal transport map from $\nu$ toward $\mu$ w.r.t. $W_A$). Since for any $z\in \R^2$ we have
	\[
	\nu_1 |z|^2\leq |z|^2_A\leq \nu_2 |z|^2
	\]
	\begin{align*}
	W^2_A(\mu,\nu)=\int_{\R^2}|\mt'(x,v)-(x,v)|_A^2\nu(dx,dv)\geq \nu_1 \int_{\R^2}|\mt'(x,v)-(x,v)|^2\nu(dx,dv)&\geq \nu_1 \int_{\R^2}|\mt(x,v)-(x,v)|^2\nu(dx,dv)\\
	&=\nu_1 W_2^2(\mu,\nu),
	\end{align*}
	and the converse comes with a similar argument.
\end{proof}
\subsection{Main result}
We are now in position to state the main result of the paper. \newline

We begin with some smallness assumption on the perturbation, which is consistent with $WJ$ type of inequality (see for instance \cite[Definition 3.1]{BGG1}). \newline 
More precisely, we assume that the perturbation $\psi$ is compactly supported on $[-R,R]$ for some $R>0$, satisfying the following smallness assumption
\bq
\label{eq:cond}
\begin{aligned}
&2\|\psi'\|_{L^\infty}<\alpha<\|\psi''\|_{L^\infty}<10^{-1}, \ \gamma:=\|\psi''\|_{L^\infty}-\alpha,\ c_*>2b_*,\\
&b_*:=(\gamma+\|\psi''\|^2_{L^\infty})\frac{2}{1+\sqrt{1-2(\gamma+\|\psi''\|_{L^\infty})}},\\
&c_*=\frac{1}{20} e^{ -\frac{\alpha}{2} \lt(R+2\rt)^2}e^{-\|\psi\|_{L^\infty}} \lt(1 \wedge  \frac{(\alpha -2\|\psi'\|_{L^\infty})}{4} \rt).
\end{aligned}
\eq
This condition is quite technical, and rather complicated to check. But at this cost, it enables a simple choice of the coefficients of the matrix A, in a manner of speaking. Nevertheless, this restriction is non void (see Lemma \ref{lem:novoid}). The part of this restriction which must be emphasized, is that $\gamma=\|\psi''\|_{L^\infty}-\alpha>0$, so that $U''$ is non positive, and classical coupling method fails for quadratic optimal transport metrics.  \newline
The confinement potential being set, we define the vector field $B$ as
\[
B(x,v)=\begin{pmatrix}
v\\ 
-U'(x)-v
\end{pmatrix} =\begin{pmatrix}
v\\ 
-\alpha x-\psi'(x)-v
\end{pmatrix}.
\]
For $f,g\in \PP_2(\R^2)$, and $\varphi$ the unique (up to an additive constant) convex potential such that $A^{-1}\nabla \varphi$ transports $g$ onto $f$, optimally w.r.t. $W_A$ (as given by Theorem \ref{thm:Bre} above) ,we introduce now the key functional 	
\bq
\begin{aligned}
	\JJ_A(f|g):=&-\int_{\R^2}\lal B(A^{-1}\nabla \varphi(x,v))- B(x,v),A^{-1}\nabla \varphi(x,v)-(x,v) \ral_A  g(x,v)dxdv\\
	&+\int_{\R^2} (\pa_v^2\varphi)^{-1} \lt( \lt( \pa_v^2\varphi-c   \rt)^2+\frac{\lt(b \lt( \pa_v^2\varphi-c   \rt)-c \lt( \pa_{x,v}^2\varphi-b  \rt)  \rt)^2}{\det(\nabla^2 \varphi)} \rt) g(x,v)dxdv.
\end{aligned}
\eq
The main feature of this functional is given in the 
\begin{proposition}
	\label{prop:diss}
	Let $(f_t)_{t\geq 0}$, $(g_t)_{t\geq 0}$ be two solutions to \eqref{eq:PS}. Then for any symmetric positive definite matrix and any $t>0$ there holds	
	\[
	\frac{1}{2}\frac{d}{dt}W_A^2(f_t,g_t)\leq - \JJ_A(f_t|g_t).
	\]
\end{proposition}	
	This result is no more than a careful application in the kinetic case of classical results (see for instance \cite[Theorem 23.9]{VilOT}). Note that it would be possible to be even more careful and obtain an equality instead of an upper bound, but this is not required for our purpose. \newline
	
	Next we obtain a key functional inequality in the	
\begin{proposition}
	\label{prop:key}
		Assume that $U$ is of the form \eqref{eq:U}, for some $\alpha>0$ and some compactly supported on $[-R,R]$ function $\psi$, satisfying \eqref{eq:cond}. There is a symmetric positive definite matrix $A\in \MM_2(\R)$ and $\kappa>0$, such that for any probability $f\in \PP_2(\R^2)$ with smooth enough density, there holds
	\begin{align*}
	\kappa W_A^2(f,f_\infty)&\leq \JJ_A(f|f_\infty).
	\end{align*}
\end{proposition}	
This result is greatly inspired by \cite[Proposition 3.4]{BGG1}, where a similar "entropy"-"entropy dissipation" inequality is derived, between the Wasserstein 2 distance to the equilibrium and its dissipation along some (non degenerate) Fokker-Planck equation.\newline 
Let us briefly sketch the proof of this result, which structure might remind acquainted readers the "cascade" structure of the proof of the $L^2$ result.\newline
	\begin{figure}[h]
	\centering
	\includegraphics[scale=0.17]{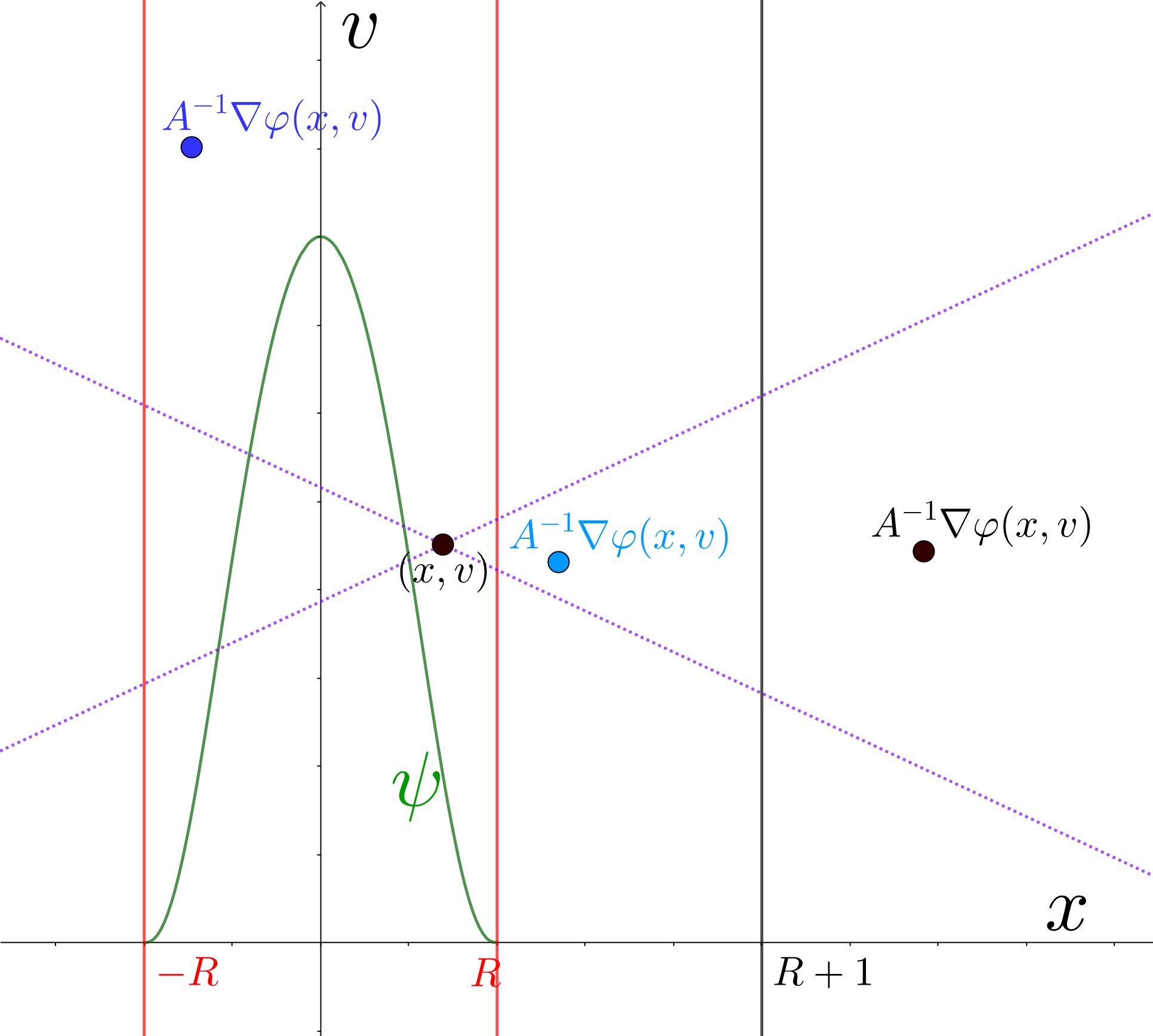}
	\caption{Decomposition of the phase space used in the proof of Proposition \ref{prop:key}}.
	\label{fig1}
\end{figure}
Frictions in velocity always provide some contraction effects in  velocity variable. Due to the local lack of convexity of $U$, we can not always benefit of contraction effects in the position variable. Of course, at $v\in \R$ fixed, if $x$ or the position component of the transported of $(x,v)$ lies far enough of the support of the perturbation $\psi$, contraction effects are obvious. \newline
Problems may arise if these two positions lie in the support of the perturbation. But since in this case, we still enjoy the contraction in velocity, we may solve the issue if the velocity component of transportation vector dominates the position component.\newline
On the remaining case, we crucially rely on the non negative quantity in the definition of $\JJ_A$, involving second order derivatives of $\varphi$. Thanks to a surprising matricial inequality (see Lemma \ref{lem:key2})), a simple geometric interpretation of this quantity is given, and with help of the contraction already found in previous regions, some contraction effects in the position variable are derived on the set of non convexity of $U$. We emphasize that, the quantity involving second order derivative in $\JJ_A$ is obtained because on the one hand we rely on the optimal transport map for quadratic cost function, thus admitting some gradient structure by Brenier's Theorem, and on the other hand by taking into account the effects of the diffusion in velocity. And optimal transport theory and hypocoercivity theory are mixed in some sense. \newline

A simple application of these two propositions, together with Gronwall's inequality yields to the
\begin{theorem}
	\label{thm:main1}
	Assume that $U$ is of the form \eqref{eq:U}, for some $\alpha>0$ and some compactly supported on $[-R,R]$ function $\psi$, satisfying \eqref{eq:cond}. For some $c\in (2b_*,c_*)$, let
	\[
	A=c\begin{pmatrix}
	\alpha+\frac{1}{2}&\frac{1}{2}  \\ 
\frac{1}{2}	& 1
	\end{pmatrix} 
	\]
	 There is $\kappa>0$ depending only on $U$ such that for any $f_0\in \PP_2(\R^2)$ it holds for any $t>0$ 
	\[
	W_A(f_t,f_\infty)\leq e^{-\kappa t} 	W_A(f_0,f_\infty),
	\]
	where $(f_t)_{t\geq 0}$ is the solution to \eqref{eq:PS} starting from $f_0$.
\end{theorem}

\subsection{Discussion}

We make here some comments about the main result of the paper. \newline

A first simple consequence of Theorem \ref{thm:main1} and Proposition \ref{prop:equ}, is that there is a constant $C_\alpha>1$ such that
\bq
\label{eq:cor}
W_2(f_t,f_\infty)\leq C_\alpha e^{-\kappa t} W_2(f_0,f_\infty).
\eq
In \cite{DET}, the authors prove that on a torus in position with no confinement, such an inequality can not hold with $C_\alpha=1$ and is obtained for some $C_\alpha>1$. It could be of some interests to investigate whether the techniques developed here extend to the case of a torus, to obtain a (slightly) stronger contraction in an equivalent Wasserstein metric.  \newline

It is important to note that Theorem \ref{thm:main1} only provides stability around equilibrium. It is a consequence of the fact that our analysis (more precisely Proposition \ref{prop:key}) strongly relies on the explicit knowledge of the density of the stationary measure. As noted in the introduction, a consequence of \cite{BGGKin} is that, under the assumption of strict convexity of $U$, there is a symmetric positive definite matrix $A$ and $\kappa>0$ such that for any $f_0,g_0$
\[
W_A(f_t,g_t)\leq e^{-\kappa t}W_A(f_0,g_0).
\]
For a non degenerated diffusion, it is well known that the strict convexity of the potential is a necessary and sufficient condition for a contraction inequality in usual Wasserstein-2 metric (see \cite[Remark 3.6]{BGG1}). Therefore it could be interesting to wonder if the converse to the result of Bolley et al is true, that is whether there can exist a matrix $A$ such that the kinetic Fokker-Planck equation is a contractive in $W_A$ distance, only if the potential $U$ is strictly convex.   \newline

Under assumption \eqref{eq:cond}, it is clear in view of \cite[Theorem 9.9]{VilOT}, that $f_\infty$ satisfies a $(1\wedge \alpha e^{-2\|\psi\|_{L^\infty}})$-log-Sobolev inequality (as it is some bounded perturbation of some strictly convex function). Hence for any $f_0\in \PP_2(\R^2)$ there holds
\[
	W_2^2(f_0,f_\infty)\leq   	\frac {2}{1\wedge \alpha e^{-2\|\psi\|_{L^\infty}}}\HH(f_0|f_\infty).
\]
On the other hand, we can use some regularization result (see for instance \cite[Proposition 15]{GM}), to obtain some constant $C_U>0$ depending only on $U$ such that for any $t_0\in (0,1)$ and $t>t_0$,
\begin{align*}
\HH(f_t|f_\infty)&\leq C_U t_0^{-3}W^2_2(f_{t-t_0},f_\infty).
\end{align*}
Using then together \eqref{eq:cor} and the above remark yields
\begin{align*}
\HH(f_t|f_\infty)&\leq C_{U}C_\alpha t_0^{-3} e^{-2\kappa(t-t_0)}W^2_2(f_{0},f_\infty)\\
& \leq C_U\frac{e^{-2\kappa(t-t_0)}}{t_0^3}\HH(f_{0}|f_\infty),
\end{align*}
which is \eqref{eq:HVill} (up to the value of the constants on which we will not comment here). Therefore Theorem \ref{thm:main1} is in some sense "stronger" than \cite[Theorem 39]{VilHypo} for the specific case on which we are focusing. The relation between these results may yield to think that there is here some common structure to understand, as indicated in \cite[Chapter 6]{VilHypo}. \newline

Finally, it can safely by conjectured that under assumption \eqref{eq:cond}, $(e^{-U},U')$ satisfies a WJ inequality in the sense of \cite[Definition 3.1]{BGG1}. The result can be obtained either by checking that assumption \eqref{eq:cond} implies the sufficient condition of \cite[Proposition 3.8]{BGG1}, either by straightforward computations. It could be interesting to obtain Theorem \ref{thm:main1} under the assumption that $(e^{-U},U')$ satisfies a WJ inequality, in the vein of the entropy method of \cite[Theorem 37]{VilHypo} which is based on the assumption that $e^{-U}$ satisfies a log-Sobolev inequality, or the $L^2$ method based on Poincaré inequality. Another possible direction of investigation, would be to obtain a more abstract and general result, based on commutator theory, in the spirit of \cite[Theorem 28]{VilHypo}.\newline
These questions are delayed to some future works, as is the possible extension to larger dimension (which seems mostly technical), or the extension to non-linearity and particles systems.

	\section{Proof of Theorem \ref{thm:main1}}

\subsection{Proof of Proposition \ref{prop:diss}}

Let $f_0\in \PP_2(\R^2)$ and $(f_t)_{t\geq 0}$ be the unique solution to \eqref{eq:PS} for the initial condition $f_0$. In view of \cite[Theorem A.15, Theorem A.19]{VilHypo} it admits for any $t>0$ a smooth non negative density. Therefore we may take the logarithm, and observe that, actually $(f_t)_{t\geq 0}$ solves the transport equation
\bq
\label{eq:trans}
\pa_t f_t+\nabla_{x,v}\cdot \lt( \xi_t  f_t \rt)=0,
\eq
with advection field
\[
\xi_t(x,v):=\begin{pmatrix}
v\\ 
- v-U'(x) -\pa_v \ln f_t(x,v)
\end{pmatrix}= B(x,v)- \begin{pmatrix}
0\\ 
\pa_v \ln f_t(x,v).
\end{pmatrix}
\]

$\diamond$ Step one : \newline

We first check that for any $\sigma\in \PP_2(\R^2)$ there holds
\begin{align*}
\frac{1}{2}\frac{d}{dt}W_A^2(f_t,\sigma)&\leq \int_{\R^{2}}A(A^{-1}\nabla \varphi_t(x,v)-(x,v))\cdot \xi_t(A^{-1}\nabla \varphi_t(x,v) )\sigma(x,v)dxdv.\\
&=\int_{\R^{2}}A(A^{-1}\nabla \varphi_t(x,v)-(x,v))\cdot B(A^{-1}\nabla \varphi_t(x,v) )\sigma(x,v)dxdv\\
&-\int_{\R^{2}}A(A^{-1}\nabla \varphi_t(x,v)-(x,v))\cdot\begin{pmatrix}
0\\ 
-\pa_v \ln f_t(A^{-1}\nabla \varphi_t (A(x,v))
\end{pmatrix} \sigma(x,v)dxdv\\
&=I_1-I_2.
\end{align*}
Indeed, let $t>h>0$ and denote $\varphi_{t}$ the unique (up to an additive constant) convex potential such that $A^{-1}\nabla \varphi_t$ transport $\sigma$ onto $f_t$ optimally w.r.t. $W_A$, so that
\[
W_A^2(f_t,\sigma)=\int_{\R^2} |(x,v)-A^{-1}\nabla \varphi_t(x,v)|_A^2\sigma(dx,dv).
\]
We denote $\Xi_t^{t+h}(z)$ the solution at time $t+h$ to the ODE
\[
\frac{d}{ds} \Xi_t^{t+s}=\xi_s(\Xi_t^{t+s}), \ \Xi_t^{t}(z)=z,
\]
so that $\Xi_t^{t+h}\circ A^{-1}\nabla \varphi_t$ transports, not optimally, $\sigma$ onto $f_{t+h}$ and 
\[
\frac{W_A^2(f_{t+h},\sigma)-W_A^2(f_t,\sigma)}{2h}\leq \int_{\R^2} \frac{|(x,v)-\Xi_t^{t+h}\circ A^{-1}\nabla \varphi_t(x,v)|_A^2-|(x,v)-A^{-1}\nabla \varphi_t(x,v)|_A^2}{2h}\sigma(dx,dv).
\]
We conclude this step by observing that by definition of $\Xi_t^{t+h}$ there holds
\[
\lim_{h\rightarrow 0} \frac{|(x,v)-\Xi_t^{t+h}\circ A^{-1}\nabla \varphi_t(x,v)|_A^2-|(x,v)-A^{-1}\nabla \varphi_t(x,v)|_A^2}{2h}=A(A^{-1}\nabla \varphi_t(x,v)-(x,v))\cdot \xi_t(A^{-1}\nabla \varphi_t(x,v) ).
\]

$\diamond$ Step two : \newline
We choose $\sigma=f_{\infty}$. Since for any $(x,v)$ there holds 
\bq
\label{eq:condinv}
(x,v)=A^{-1}\nabla \varphi_t(\nabla \varphi^*_t(A(x,v)),
\eq
by point $(iii)$ of Theorem \ref{thm:Bre} and since $A^{-1}\nabla \varphi_t$ transports $f_{\infty}$ to $f_t$, we have
\begin{align*}
I_2=&\int_{\R^{2}}A((x,v)-\nabla \varphi_t^*(A(x,v)))\cdot\begin{pmatrix}
0\\ 
-\pa_v \ln f_t(x,v)
\end{pmatrix} f_t(x,v)dxdv=\\
=&\int_{\R^{2}}((x,v)-\nabla \varphi_t^*(A(x,v)))\cdot\begin{pmatrix}
-b\pa_v \ln f_t(x,v)\\ 
-c\pa_v \ln f_t(x,v)
\end{pmatrix} f_t(x,v)dxdv\\
=&\int_{\R^{2}}-b(x-\pa_x \varphi_t^*(A(x,v)))\pa_v f_t(x,v)dxdv+\int_{\R^{2}}-c(v-\pa_v \varphi_t^*(A(x,v)))\pa_v f_t(x,v)dxdv.\\
\end{align*}
Then integrating each term by parts, using \eqref{eq:condinv}
 and the fact that $\nabla \varphi^*_t(A(x,v))$ transports $f_t$ onto $f_\infty$  yields
\begin{align*}
I_2&=\int_{\R^{2}}\lt(-b^2\pa^2_x \varphi_t^*(A(x,v)) -cb\pa^2_{v,x} \varphi_t^*(A(x,v)) \rt) f_t(x,v)dxdv\\
&+\int_{\R^{2}}c(1-b\pa^2_{x,v} \varphi_t^*(A(x,v)) -c\pa^2_v \varphi_t^*(A(x,v))) f_t(x,v)dxdv\\
&=\int_{\R^{2}}\lt(-b^2\pa^2_x \varphi_t^*(AA^{-1}\nabla \varphi_t(\nabla \varphi^*_t(A(x,v))) -cb\pa^2_{v,x} \varphi_t^*(AA^{-1}\nabla \varphi_t(\nabla \varphi^*_t(A(x,v))) \rt) f_t(x,v)dxdv\\
&+\int_{\R^{2}}c(1-b\pa^2_{x,v} \varphi_t^*(AA^{-1}\nabla \varphi_t(\nabla \varphi^*_t(A(x,v))) -c\pa^2_v \varphi_t^*(AA^{-1}\nabla \varphi_t(\nabla \varphi^*_t(A(x,v)))) f_t(x,v)dxdv\\
&=\int_{\R^{2}}\lt(-b^2\pa^2_x \varphi_t^*(\nabla \varphi_t(x,v)) -cb\pa^2_{v,x} \varphi_t^*(\nabla \varphi_t(x,v))\rt) f_\infty(x,v)dxdv\\
&+\int_{\R^{2}}c(1-b\pa^2_{x,v} \varphi_t^*(\nabla \varphi_t(x,v)) -c\pa^2_v \varphi_t^*(\nabla \varphi_t(x,v))) f_\infty(x,v)dxdv.
\end{align*}
By point $(iii)$ of Theorem \ref{thm:Bre}, $\varphi_t$ is $\CC^2$ and $\nabla \varphi_t^*(\nabla \varphi_t(x,v))=(x,v)$ we obtain by differentiation
\[
\nabla^2 \varphi^*_t(\nabla \varphi_t(x,v))=\lt( \nabla^2 \varphi_t(x,v) \rt)^{-1}=\begin{pmatrix}
 \pa_x^2 \varphi_t& \pa_{x,v}^2 \varphi_t\\ 
\pa_{v,x}^2 \varphi_t & \pa_{v}^2 \varphi_t
\end{pmatrix}^{-1}= \frac{1}{\det(\nabla^2 \varphi_t(x,v))} \begin{pmatrix}
 \pa_v^2 \varphi_t& -\pa_{x,v}^2 \varphi_t\\ 
-\pa_{v,x}^2 \varphi_t & \pa_{x}^2 \varphi_t
\end{pmatrix}.
\]
Therefore, identifying term by term yields
\begin{align*}
I_2=- \int_{\R^{2}} \frac{1}{\det(\nabla^2 \varphi_t(x,v))}   \lt(b^2\pa^2_v \varphi_t(x,v) -2cb\pa^2_{x,v} \varphi_t(x,v)+c^2\pa^2_x \varphi_t(x,v)-c\rt) f_\infty(x,v)dxdv.
\end{align*}

$\diamond$ Step three : \newline

Next observe that since $f_\infty$ is a stationary solution to \eqref{eq:trans}, we easily find that for any $t>0$
\begin{align*}
0&= \int_{\R^{2}}|A^{-1}\nabla \varphi_t(x,v)-(x,v)|^2_A\nabla_{x,v}\cdot \lt(\xi_\infty f_\infty\rt)dxdv\\
&= \int_{\R^{2}}A(A^{-1}\nabla \varphi_t(x,v)-(x,v))\cdot \xi_\infty(x,v)f_\infty(x,v)dxdv.\\
&=\int_{\R^{2}}A(A^{-1}\nabla \varphi_t(x,v)-(x,v))\cdot B(x,v)f_\infty(x,v)dxdv\\
&-\int_{\R^{2}}A(A^{-1}\nabla \varphi_t(x,v)-(x,v))\cdot\begin{pmatrix}
0\\ 
-\pa_v \ln f_\infty(x,v)
\end{pmatrix} f_\infty(x,v)dxdv\\
&=J_1-J_2.
\end{align*}
By integration by parts, we obtain
\begin{align*}
J_2&=-\int_{\R^2}(\pa_v \varphi(x,v)-(bx+cv))\pa_v f_\infty(x,v)dxdv\\
&=-\int_{\R^2}(c -\pa^2_v \varphi(x,v)) f_\infty(x,v)dxdv. 
\end{align*}

$\diamond$ Step three : \newline
Gathering all the above results we have

\begin{align*}
\frac{1}{2}\frac{d}{dt}W_A^2(f_t,f_\infty)&\leq\int_{\R^{2}}A(A^{-1}\nabla \varphi_t(x,v)-(x,v))\cdot \lt(B(A^{-1}\nabla \varphi_t(x,v)-B(x,v)\rt) )f_\infty(x,v)dxdv\\
&+\int_{\R^{2}} \frac{1}{\det(\nabla^2 \varphi_t(x,v))}   \lt(b^2\pa^2_v \varphi_t(x,v) -2cb\pa^2_{x,v} \varphi_t(x,v)+c^2\pa^2_x \varphi_t(x,v)-c\rt) f_\infty(x,v)dxdv\\
&+\int_{\R^2}(\pa^2_v \varphi(x,v)-c) f_\infty(x,v)dxdv.
\end{align*}
Now, note that
\begin{align*}
\frac{c^2\pa^2_x \varphi_t}{\det(\nabla^2 \varphi_t)}  -\frac{c^2}{\pa^2_v \varphi_t} &=c^2 \lt(  \frac{1}{\pa^2_v \varphi_t-\frac{(\pa^2_{x,v} \varphi_t)^2}{\pa^2_x \varphi_t}   } -\frac{1}{\pa^2_v \varphi_t}\rt)\\
&=c^2\frac{\frac{(\pa^2_{x,v} \varphi_t)^2}{\pa^2_x \varphi_t}  }{\pa^2_v \varphi_t \lt(\pa^2_v \varphi_t-\frac{(\pa^2_{x,v} \varphi_t)^2}{\pa^2_x \varphi_t}  \rt) }\\
&=c^2\frac{ (\pa^2_{x,v} \varphi_t)^2  }{\pa^2_v \varphi_t \det(\nabla^2 \varphi_t))  },
\end{align*}
and the result is proved since
	\begin{align*}
	&\frac{1}{\det(\nabla^2 \varphi_t)}   \lt(b^2\pa^2_v \varphi_t -2cb\pa^2_{x,v} \varphi_t+c^2\pa^2_x \varphi_t-c\rt) + (\pa^2_v \varphi_t(x,v)-c)\\
	&=\pa^2_v \varphi_t(x,v)-2c+ \frac{c^2}{\pa^2_v \varphi_t}+ \frac{c^2\pa^2_x \varphi_t}{\det(\nabla^2 \varphi_t)}  -\frac{c^2}{\pa^2_v \varphi_t}+\frac{b^2\pa^2_v \varphi_t -2cb\pa^2_{x,v} \varphi_t}{\det(\nabla^2 \varphi_t)}   \\
	&= \lt( \pa_v^2\varphi_t-c   \rt)^2(\pa_v^2\varphi_t)^{-1}+\frac{c^2\frac{ (\pa^2_{x,v} \varphi_t)^2  }{\pa^2_v \varphi_t   }+b^2\pa^2_v \varphi_t -2cb\pa^2_{x,v} \varphi_t}{\det(\nabla^2 \varphi_t)} \\
&= \lt( \pa_v^2\varphi_t-c   \rt)^2(\pa_v^2\varphi_t)^{-1}+\frac{1}{\det(\nabla^2 \varphi_t)} \lt( b (\pa_v^2 \varphi_t)^{1/2}-c  \frac{ \pa^2_{v,x} \varphi_t }{\pa^2_v \varphi_t^{1/2}}  \rt)^2\\
&=(\pa_v^2\varphi(x,v))^{-1} \lt( \lt( \pa_v^2\varphi-c   \rt)^2+\frac{\lt(b \pa_v^2\varphi-c  \pa_{x,v}^2\varphi  \rt)^2}{\det(\nabla^2 \varphi)} \rt).
	\end{align*}

\subsection{Proof of Proposition \ref{prop:key}}

	We choose $A$ of the form
	\[
	A=\begin{pmatrix} b+c\alpha
	& b \\ 
	b & c
	\end{pmatrix},
	\]
	for $c=2b\in (2b_*,c_*)$. For $v\in \R$ we define
	\begin{align*}
	&\XX_v:=\Bigl\{x\in \R, \  \ |x|\leq R+1, \det(A)^{-1}|c\pa_x \varphi(x,v)-b\pa_v\varphi(x,v)|\leq R+1\Bigr\}\\
	&\YY_v:=\Bigl\{x \in \XX_v, \ |\det(A)^{-1}(c\pa_x \varphi(x,v)-b\pa_v\varphi(x,v))-x|\geq  |\det(A)^{-1}((b+c\alpha)\pa_v \varphi(x,v)-b\pa_x\varphi(x,v))-v| \Bigr\}.
	\end{align*}
	We decompose	
	\begin{align*}
		&-\int_{\R}\lal B(A^{-1}\nabla \varphi(x,v))- B(x,v),A^{-1}\nabla \varphi(x,v)-(x,v) \ral_Ae^{-U(x)}dx=\\
		&-\int_{\R\setminus \XX_v \cup \XX_v\setminus \YY_v \cup \YY_v}\lal B(A^{-1}\nabla \varphi(x,v))- B(x,v),A^{-1}\nabla \varphi(x,v)-(x,v) \ral_A e^{-U(x)}dx\\
		&:=\II_1+\II_2+\II_3
	\end{align*}	
	
	$\diamond$ Estimate of $\II_1$ and $\II_2$: \newline
	
	Using $(i)$ and $(ii)$ of Lemma \ref{lem:key1} with $z_1=\nabla \varphi(x,v)$ and $z_2=A\begin{pmatrix}
	x\\ 
	v
\end{pmatrix}$ we have that for any $x\in \R\setminus \YY_v$
	\[
	-\int_{\R\setminus \YY_v}\lal B(A^{-1}\nabla \varphi(x,v))- B(x,v),A^{-1}\nabla \varphi(x,v)-(x,v) \ral_A e^{-U(x)}dx \geq c\min(\kappa_1,\kappa_2) \int_{\R\setminus \YY_v} |A^{-1}\nabla \varphi(x,v)-(x,v)|^2e^{-U(x)}dx
	\]
	$\diamond$ Estimate of $\II_3$: \newline
 	For each $v\in \R$ we define	
	\[
	x_v^+=\sup \YY_v,
	\]
	and $x'\in \lt[	x_v^+,R+2\rt]$ (which is not necessarily unique) such that
	\[
	\lt|A^{-1}\nabla \varphi(x',v)-(x',v)\rt|=\inf_{y\in \lt[	x_v^+,R+2\rt]}\lt|A^{-1}\nabla \varphi(y,v)-(y,v)\rt|.
	\]
	For any $x\in \YY_v$, there holds by Taylor's expansion	
	\begin{align*}
	\pa_v \varphi(x,v)-(bx+cv)=\pa_v \varphi(x',v)-(bx'+cv)+\int_0^1 (\pa_{x,v}^2\varphi-b)(x_s,v) (x-x') ds,
	\end{align*}
	where $x_s=sx+(1-s)x'$. Therefore
	\begin{align*}
		|	\pa_v \varphi(x,v)-(bx+cv)|^2&\leq 2|	\pa_v \varphi(x',v)-(bx'+cv)|^2+2\lt|\int_0^1 (\pa_{x,v}^2\varphi-b)(x_s,v) (x-x') ds\rt|^2:=K_1+K_2,
	\end{align*}
$\bullet$ Estimate of $K_1$ : \newline\\
Since for any $z\in \R^2$ we have
\[
|Az|^2\leq \rho(A)|z|^2_A\leq \rho(A)^2 |z|^2,
\]
we obtain by definition of $x',x_v^+$
\begin{align*}
|\pa_v \varphi(x',v)-(bx'+cv)|^2&\leq 	\lt|\nabla \varphi(x',v)-A(x',v)\rt|^2=\lt|A\lt(A^{-1}\nabla \varphi(x',v)-(x',v)\rt)\rt|^2\\
&\leq \rho(A)^2\lt|A^{-1}\nabla \varphi(x',v)-(x',v)\rt|^2\\
& \leq \rho(A)^2\int_{R+1}^{R+2}  \lt|A^{-1}\nabla \varphi(x,v)-(x,v)\rt|^2 dx\\
&\leq  \rho(A)^2e^{\sup{y\in \lt[	R+1,R+2\rt]}U(y)}\int_{R+1}^{R+2} \lt|A^{-1}\nabla \varphi(x,v)-(x,v)\rt|^2 e^{-U(x)}dx\\
&\leq \rho(A)^2e^{ \frac{\alpha}{2} \lt(R+2\rt)^2}e^{\|\psi\|_{L^\infty}}\int_{\R\setminus \YY_v} \lt|A^{-1}\nabla \varphi(x,v)-(x,v)\rt|^2 e^{-U(x)}dx
\end{align*}
$\bullet$ Estimate of $K_2$ : \newline\\
By Cauchy-Schwarz's inequality we have 
\begin{align*}
		\lt|\int_0^1 (\pa_{x,v}^2\varphi-b)(x_s,v) (x-x') ds\rt|^2		&=\lt|\int_0^1\frac{\pa_{x,v}^2\varphi-b }{\pa_x^2\varphi^{1/2}}(x_s,v)e^{-\frac{U(x_s)}{2}} e^{\frac{U(x_s)}{2}}(\pa_x^2\varphi(x_s,v))^{1/2}(x-x') ds\rt|^2\\
		&\leq \lt(\int_0^1\frac{(\pa_{x,v}^2\varphi-b)^2}{\pa_x^2\varphi}(x_s,v)e^{-U(x_s)} ds\rt) \lt(\int_0^1\pa_x^2\varphi(x_s,v)e^{U(x_s)}(x-x')^2 ds\rt).
	\end{align*}
	 Since $x,x'\in \lt[-\lt(R+2\rt);\lt(R+2\rt)\rt]$ we have	
	\begin{align*}
		&\int_0^1\pa_x^2\varphi(x_s,v)e^{U(x_s)}(x-x')^2 ds\leq e^{ \frac{\alpha}{2} \lt(R+2\rt)^2}e^{\|\psi\|_{L^\infty}}(\pa_x \varphi(x,v)-\pa_x \varphi(x',v))(x-x').
	\end{align*}
		Next observe that
	\begin{align*}
		\lt|\pa_x \varphi(x,v)-\pa_x \varphi(x',v)\rt|&
		=\lt|(\pa_x \varphi(x,v)-(bx+cv))-(\pa_x \varphi(x',v)-(b x'+cv))+b(x'-x)\rt|\\
		&\leq \lt|\nabla \varphi(x,v)-A(x,v)\rt|+\lt|\nabla \varphi(x',v)-A(x',v)\rt| + b(|x|+|x'|) \\
		&\leq \rho(A)\lt( \lt|A^{-1}\nabla \varphi(x,v)-(x,v)\rt|+\lt|A^{-1}\nabla \varphi(x',v)-(x',v)\rt|  \rt)+b(|x|+|x'|)\\
		&\leq    \rho(A)\lt( \lt|A^{-1}\nabla \varphi(x,v)-(x,v)\rt|+\lt|A^{-1}\nabla \varphi(x^*_v,v)-(x^*_v,v)\rt|  \rt)  +b \lt(R+2\rt).
	\end{align*}	
But since $x\in \YY_v$ we have
	\begin{align*}
	\lt|A^{-1}\nabla \varphi(x,v)-(x,v)\rt|&=\sqrt{\lt(\det(A)^{-1}(c\pa_x \varphi-b \pa_v \varphi) -x\rt)^2+\lt(\det(A)^{-1}((b+c\alpha)\pa_v \varphi-b \pa_x \varphi) -v\rt)^2}\\
	& \leq \lt|c\pa_x \varphi-b\pa_v \varphi -x\rt|\sqrt{2}\leq \lt(R+2\rt) \sqrt{2},
\end{align*}	
and by definition of $x^*_v$  we similarly obtain
\[
\lt|A^{-1}\nabla \varphi(x^*_v,v)-(x^*_v,v)\rt|\leq \lt(R+2\rt) \sqrt{2}.
\]
Gathering all these estimates yields 
\begin{align*}
\int_0^1\pa_x^2\varphi(x_s,v)e^{U(x_s)}(x-x')^2 ds&\leq 2 c  e^{ \frac{\alpha}{2} \lt(R+2\rt)^2}e^{\|\psi\|_{L^\infty}} \lt(R+2\rt)^2\lt( 2\frac{\rho(A)}{c}\sqrt{2}  +\frac{b}{c}  \rt)\\
\end{align*}
Therefore for any $x\in \YY_v$ there holds
	\begin{align*}
	&|\pa_v \varphi(x,v)-(bx+cv)|^2\leq 2c^{2}\lt(\frac{\rho(A)}{c}\rt)^2e^{ \frac{\alpha}{2} \lt(R+2\rt)^2}e^{\|\psi\|_{L^\infty}} \int_{\R\YY_v} \lt|A^{-1}\nabla \varphi(y,v)-(y,v)\rt|^2 e^{-U(y)}dy\\
	&+  2 c  e^{ \frac{\alpha}{2} \lt(R+2\rt)^2}e^{\|\psi\|_{L^\infty}} \lt(R+2\rt)^2\lt( 2\frac{\rho(A)}{c}\sqrt{2}  +\frac{b}{c}  \rt) \lt(\int_\R\frac{(\pa_{x,v}^2\varphi-b)^2}{\pa_x^2\varphi}(y,v)e^{-U(y)} dy\rt)\\
	&=:c^2 C_1 \int_{\R\setminus \YY_v} \lt|A^{-1}\nabla \varphi(y,v)-(y,v)\rt|^2 e^{-U(y)}dy+c C_2  \int_\R\frac{(\pa_{x,v}^2\varphi-b)^2}{\pa_x^2\varphi}(y,v)e^{-U(y)} dy
\end{align*}
Hence integrating over $\YY_v$ and using Lemma \ref{lem:key2}, we obtain
\begin{align*}
\int_{\YY_v}|\pa_v \varphi(x,v)-(bx+cv)|^2e^{-U(x)}dx\leq c^2 C_1\int_{\R\setminus \YY_v}\lt|A^{-1}\nabla \varphi(y,v)-(y,v)\rt|^2 e^{-U(y)}dy&\\
+ c C_2 \int_\R(\pa_v^2\varphi)^{-1} \lt( \lt( \pa_v^2\varphi-c   \rt)^2+\frac{\lt(b \pa_v^2\varphi-c \pa_{x,v} \varphi \rt)^2}{\det(\nabla^2 \varphi)} \rt)(x,v)e^{-U(x)} dx.
\end{align*}

	$\diamond$ Conclusion : \newline
	By definition of $\JJ_A$ we have
		\begin{align*}
	Z\JJ_A(f|f_\infty&)=\int_{\R} \lt(\int_{\R\setminus \YY_v \cup \YY_v} -\lal B(A^{-1}\nabla \varphi(x,v))- B(x,v),A^{-1}\nabla \varphi(x,v)-(x,v) \ral_A  e^{-U(x)}dx\rt)e^{-\frac{|v|^2}{2}}dv\\
	&+\int_{\R}\lt(\int_{\R} (\pa_v^2\varphi)^{-1} \lt( \lt( \pa_v^2\varphi-c   \rt)^2+\frac{\lt(b \lt( \pa_v^2\varphi-c   \rt)-c \lt( \pa_{x,v}^2\varphi-b  \rt)  \rt)^2}{\det(\nabla^2 \varphi)} \rt) e^{-U(x)}dx\rt)e^{-\frac{|v|^2}{2}}dv.\\
	\end{align*}	
	Using the above steps yields	
	\begin{align*}
	Z\JJ_A(f|f_\infty)&\geq  c \frac{(\kappa_1\wedge \kappa_2)}{2} \int_{\R} \lt(  \int_{\R\setminus \YY_v} |A^{-1}\nabla \varphi(x,v)-(x,v)|^2 e^{-U(x)}dx\rt)e^{-\frac{|v|^2}{2}}dv\\
	&+c \frac{(\kappa_1\wedge \kappa_2)}{2} \int_{\R} \lt(  \int_{\R\setminus \YY_v} |A^{-1}\nabla \varphi(x,v)-(x,v)|^2 e^{-U(x)}dx\rt)e^{-\frac{|v|^2}{2}}dv\\
	&+\int_{\R}\lt(\int_{\R} (\pa_v^2\varphi)^{-1} \lt( \lt( \pa_v^2\varphi-c   \rt)^2+\frac{\lt(b \lt( \pa_v^2\varphi-c   \rt)-c \lt( \pa_{x,v}^2\varphi-b  \rt)  \rt)^2}{\det(\nabla^2 \varphi)} \rt)e^{-U(x)}dx\rt)e^{-\frac{|v|^2}{2}}dv\\
	&+\int_{\R} \lt(\int_{ \YY_v} -\lal B(A^{-1}\nabla \varphi(x,v))- B(x,v),A^{-1}\nabla \varphi(x,v)-(x,v) \ral_A  e^{-U(x)}dx\rt)e^{-\frac{|v|^2}{2}}dv.
	\end{align*}
	Since $c<c_*$ there holds
	\begin{align*}
	cC_1\leq \frac{(\kappa_1\wedge \kappa_2)}{2}, cC_2\leq 1, 
	\end{align*}
	therefore, by the above steps,
	\begin{align*}
	\int_{\R} \lt( \int_{\YY_v} |\pa_v \varphi(x,v)-(bx+cv)|^2 e^{-U(x)}dx\rt)e^{-\frac{|v|^2}{2}}dv\leq  c \frac{(\kappa_1\wedge \kappa_2)}{2} \int_{\R} \lt(  \int_{\R\setminus \YY_v} |A^{-1}\nabla \varphi(x,v)-(x,v)|^2 e^{-U(x)}dx\rt)e^{-\frac{|v|^2}{2}}dv&\\
	+\int_{\R}\lt(\int_{\R} (\pa_v^2\varphi)^{-1} \lt( \lt( \pa_v^2\varphi-c   \rt)^2+\frac{\lt(b \lt( \pa_v^2\varphi-c   \rt)-c \lt( \pa_{x,v}^2\varphi-b  \rt)  \rt)^2}{\det(\nabla^2 \varphi)} \rt)e^{-U(x)}dx\rt)e^{-\frac{|v|^2}{2}}dv,&
	\end{align*}
	and
	\begin{align*}
	Z\JJ_A(f|f_\infty)&\geq c \frac{(\kappa_1\wedge \kappa_2)}{2}\int_{\R} \lt(  \int_{\R\setminus \YY_v} |A^{-1}\nabla \varphi(x,v)-(x,v)|^2 e^{-U(x)}dx\rt)e^{-\frac{|v|^2}{2}}dv\\
	&+\int_{\R} \lt( \int_{\YY_v} |\pa_v \varphi(x,v)-(bx+cv)|^2 e^{-U(x)}dx\rt)e^{-\frac{|v|^2}{2}}dv\\
	&+\int_{\R} \lt(\int_{ \YY_v} -\lal B(A^{-1}\nabla \varphi(x,v))- B(x,v),A^{-1}\nabla \varphi(x,v)-(x,v) \ral_A  e^{-U(x)}dx\rt)e^{-\frac{|v|^2}{2}}dv
	\end{align*}
	Finally, since $b>b_*$, we have by point $(iii)$ of Lemma \ref{lem:key1} 
		\begin{align*}
	Z\JJ_A(f|f_\infty)&\geq  c\frac{(\kappa_1\wedge \kappa_2)}{2}\int_{\R} \lt(  \int_{\R\setminus \YY_v} |A^{-1}\nabla \varphi(x,v)-(x,v)|^2 e^{-U(x)}dx\rt)e^{-\frac{|v|^2}{2}}dv\\
	&+ c \kappa_3\int_{\R} \lt(\int_{ \YY_v} |A^{-1}\nabla \varphi(x,v)-(x,v)|^2 e^{-U(x)}dx\rt)e^{-\frac{|v|^2}{2}}dv,
	\end{align*}
	and
			\begin{align*}
	\JJ_A(f|f_\infty)&\geq  c\min\lt(\frac{(\kappa_1\wedge \kappa_2)}{2},\kappa_3\rt)\int_{\R^2}  |A^{-1}\nabla \varphi(x,v)-(x,v)|^2 f_\infty(x,v)dxdv\\
	&\geq \rho(A)^{-1}c \min\lt(\frac{(\kappa_1\wedge \kappa_2)}{2},\kappa_3\rt)\int_{\R^2}  |A^{-1}\nabla \varphi(x,v)-(x,v)|_A^2 f_\infty(x,v)dxdv,
	\end{align*}
	and the result is proved with
	\[
	\kappa=\lt(\frac{\alpha+\frac{3}{2}+\sqrt{\lt(\alpha-\frac{1}{2}\rt)^2+1}}{2}\rt)^{-1}\min\lt(\frac{(\kappa_1\wedge \kappa_2)}{2},\kappa_3\rt).
	\]

\section*{Acknowledgements}
The author was supported by the Fondation Mathématique Jacques Hadamard, and warmly thanks Patrick Cattiaux and Arnaud Guillin for many advices, comments and discussions which have made this work possible.

\appendix

\section{Toolbox}
\textit{Proof of Theorem \ref{thm:Bre}} : \newline

$(i)$ Since the cost function $c(z)=|z|_A^2$, thanks to which $W_A$ is defined, is strictly convex (since $A$ is positive definite) and superlinear (i.e. $\lim_{|z|\rightarrow +\infty}\frac{|z|^2_A}{|z|}=+\infty$), we may invoke \cite[Theorem 2.44]{VilOT} (see also \cite[Theorem 1.17]{OTAM}), and for any $\mu,\nu\in \PP_2(\R^2)$ with smooth enough densities, we obtain that here exists a unique map $\mt : \R^2\mapsto \R^2$ such that $\mt\# \nu=\mu$ and
\[
W_A^2(\mu,\nu)=\int_{\R^2} |(x,v)-\mt(x,v)|^2_A\nu(x,v)dxdv.
\]
Moreover $\mt$ is given as
\[
\mt(z)=z-\nabla c^*(\nabla \Phi(z)),
\]
where $\Phi$ is some $c$-concave function (see \cite[Definition 2.33 ]{VilOT}. But since for any $z\in \R^2$, $\nabla c^*(z)=A^{-1}z$ we have
\[
\mt(z)=A^{-1}(Az-\nabla \Phi(z))=A^{-1}\nabla \lt( \frac{|z|^2_A}{2} -  \Phi(z) \rt). 
\]
We set
\[
\varphi(z)=\frac{|z|^2_A}{2}-\Phi(z),
\]
and observe that it is a convex functional since $\Phi$ is $|\cdot|^2_A$-concave (see \cite[Proposition 1.21]{OTAM}). \newline

$(iii)$ follows from the regularity theory for Monge-Ampère equation (see for instance \cite[Theorem 1]{Cor}) and the fact that $\mu$ and $\nu$ have $\CC^1$, non vanishing densities.\newline

$(ii)$ follows from point $(i)$ and the fact that, since $\forall(x,v)\in\R^2$ we have $\nabla \varphi^*(\nabla \varphi(x,v))=(x,v)$, then $(A^{-1}\nabla \varphi)^{-1}=\nabla \varphi^*(A\cdot)$ transports $\mu$ onto $\nu$, and point $(i)$.

\qed

\begin{lemma}
	\label{lem:key1}
		Assume that $U$ is of the form \eqref{eq:U}, for some $\alpha>0$ and some compactly supported on $[-R,R]$ function $\psi$, satisfying \eqref{eq:cond}.\newline
	Let $c=2b>0$,  and
	\[
	A=\begin{pmatrix} b+c\alpha
	& b \\ 
	b & c
	\end{pmatrix}.
	\]
	Assume that
	\[
	b>\frac{(\|\psi''\|_{L^\infty}+b)^2}{1+4b}+\gamma, \ \text{i.e.} \ b> b_*.
	\] 
		Then there are $\kappa_1,\kappa_2,\kappa_3>0$ such that for any $z_1=(x_1,v_1),z_2=(x_2,v_2)\in \R^{2}$, there holds, with $z'_i=(x'_i,v'_i)=A^{-1}z_i$ for $i=1,2$	
	\begin{itemize}
		\item[$(i)$] in the case $|x'_1|\geq R+1$ or $|x'_2|\geq R+1$
		\[
		-\lal B(z'_1)- B(z'_2),z'_1-z'_2 \ral_A\geq c \kappa_1 |z'_1-z'_2|^2
		\]
		with
		\[
		\kappa_1=\min \lt( \frac{1}{2} -\|\psi'\|_{L^\infty} , \frac{1}{2}\lt( \alpha -2\|\psi'\|_{L^\infty}  \rt)\rt),
		\]
		\item[$(ii)$] in the case $|x'_1|\leq R+1$ and $|x'_2|\leq R+1$ and $|x_1-x_2|<|v_1-v_2|$
		\[
		-\lal B(z'_1)- B(z'_2),z'_1-z'_2 \ral_A\geq c \kappa_2 |z'_1-z'_2|^2
		\]
		with 
		\[
		\kappa_2=\frac{1}{2} \lt( \frac{1}{2}-\|\psi''\|_{L^\infty}-\frac{\gamma}{2}\rt),
		\]
		\item[$(iii)$] in the case $|x_1|\leq R+1$ and $|x_2|\leq R+1$ and $|x_1-x_2|\geq |v_1-v_2|$ 
		\[
		-\lal B(z'_1)- B(z'_2),z'_1-z'_2 \ral_A+|v_1-v_2|^2\geq \kappa_3 |z'_1-z'_2|^2.
		\]

	\end{itemize}

\end{lemma}

\begin{proof}
	
	Recall that $U'(x)=\alpha x+\psi'(x)$ and observe that

	\begin{align*}
\lal B(z'_1)-B(z'_2),z'_1-z'_2 \ral_A=&(b+c\alpha)(v'_1-v'_2) (x'_1-x'_2)+b(v'_1-v'_2)(v'_1-v'_2)-b\lt(  (v'_1-v'_2)+( U'(x'_1)- U'(x'_2))  \rt)(x'_1-x'_2)\\
&-c\lt(  (v'_1-v'_2)+( U'(x'_1)- U'(x'_2))  \rt)(v'_1-v'_2)\\
&=-(c-b)|v'_1-v'_2|^2-b\alpha|x'_1-x'_2|^2\\
&-b(\psi'(x'_1)- \psi'(x'_2))(x'_1-x'_2)-c( \psi'(x'_1)- \psi'(x'_2))(v'_1-v'_2)
\end{align*}
$\diamond$ Proof of $(i)$ \newline

		W.l.o.g. assume that $|x'_1|\geq R+1$, and $|x'_2|\leq R$ (otherwise the result is obvious since $\psi'$ is compactly supported on $[-R,R]$)\newline
		Since
		\[
		|\psi'(x'_1)-\psi'(x'_2)|\leq |\psi'(x'_1)|\leq \|\psi'\|_{L^\infty}\leq \|\psi'\|_{L^\infty}|x'_1-x'_2|,
				\]
	by Young's inequality 
		\begin{align*}
		-\lal B(z'_1)-B(z'_2),z'_1-z'_2 \ral_A&\geq \lt(c-b\rt)|v'_1-v'_2|^2+ b\lt(\alpha-\|\psi'\|_{L^\infty}\rt)|x'_1-x'_2|^2-c\|\psi'\|_{L^\infty}|x'_1-x'_2||v'_1-v'_2|\\
		&=c\lt( \frac{1}{2}|v'_1-v'_2|^2+\frac{1}{2} \lt(\alpha-\|\psi'\|_{L^\infty}\rt) |x'_1-x'_2|^2-2\|\psi'\|_{L^\infty}|x'_1-x'_2||v'_1-v'_2|    \rt)\\
		&\geq  c \lt(   \lt( \frac{1}{2} -\|\psi'\|_{L^\infty} \rt) |v'_1-v'_2|^2 + \frac{1}{2}\lt( \alpha -2\|\psi'\|_{L^\infty}  \rt) |x'_1-x'_2|^2   \rt),
		\end{align*}
		and the result is proved with 
		\[
		\kappa_1=\min \lt( \frac{1}{2} -\|\psi'\|_{L^\infty} , \frac{1}{2}\lt( \alpha -2\|\psi'\|_{L^\infty}  \rt)\rt),
		\]
		due to condition \eqref{eq:cond}
		
	$\diamond$ Proof of $(ii)$ \newline

	Since $|x'_1-x'_2|\leq |v'_1-v'_2|$	
	
	\begin{align*}
	-\lal B(z'_1)-B(z'_2),z'_1-z'_2 \ral_A&\geq 
	\lt(c-b\rt)|v'_1-v'_2|^2-\gamma b|x'_1-x'_2|^2-c\|\psi''\|_{L^\infty}|x'_1-x'_2| |v'_1-v'_2|\\
	&\geq c\lt( \frac{1}{2}-\|\psi''\|_{L^\infty}-\frac{\gamma}{2} \rt)|v'_1-v'_2|^2,
	\end{align*}
and the result is proved with
\[
\kappa_2= \frac{1}{2} \lt( \frac{1}{2}-\|\psi''\|_{L^\infty}-\frac{\gamma}{2}\rt).
\]
	$\diamond$ Proof of $(iii)$ \newline
	
We denote $r=x_1-x_2,s=v_1-v_2$ and $r'=x'_1-x'_2, s'=v'_1-v'_2$. Since
	\[
	A\begin{pmatrix}
	r'\\ 
	s'
	\end{pmatrix}=	\begin{pmatrix}
	r\\ 
	s
	\end{pmatrix} =\begin{pmatrix}
	ar'+bs'\\ 
	br'+cs'
	\end{pmatrix},
	\]
	where $a=b+c\alpha$. We have
	\begin{align*}
	-\lal B(z'_1)-B(z'_2),z'_1-z'_2 \ral_A+|v_1-v_2|^2&\geq 
	\lt(c-b\rt)|s'|^2-\gamma b|r'|^2-c\|\psi''\|_{L^\infty}|r'| |s'|+b^2|r'|^2+2bcr's'+c^2 |s'|^2\\
	&\geq (c-b+c^2)|s'|^2+b (b-\gamma)|r'|^2-\lt(2bc+c\|\psi''\|_{L^\infty}\rt)|r'||s'|\\
	&=c \lt(\lt(\frac{1}{2}+c\rt)|s'|^2 +\frac{1}{2} (b-\gamma) |r'|^2 -(\|\psi''\|_{L^\infty}+b) |r'||s'| \rt).
	\end{align*}
	By Young's inequality
	\begin{align*}
	-(\|\psi''\|_{L^\infty}+b) |r'||s'|&=(\|\psi''\|_{L^\infty}+b) \sqrt{\frac{2\lt(\frac{1}{2}+c\rt)}{\|\psi''\|_{L^\infty}+b}}|s'|\sqrt{\frac{\|\psi''\|_{L^\infty}+b}{2\lt(\frac{1}{2}+c\rt)}}|r'|\\
	&\geq -\lt(\frac{1}{2}+c\rt)|s'|^2-\frac{(\|\psi''\|_{L^\infty}+b)^2}{4\lt(\frac{1}{2}+c\rt)}|r'|^2,
	\end{align*}
	and	provided that
	\[
	\frac{1}{2}(b-\gamma)-\frac{(\|\psi''\|_{L^\infty}+b)^2}{2+4b}>0,
	\] 
	the result is proved with
	\[
	\kappa_3=\frac{1}{2}\lt(\frac{1}{2}(b-\gamma)-\frac{(\|\psi''\|_{L^\infty}+b)^2}{2+4b}\rt)
	\]

\end{proof}

\begin{lemma}
	\label{lem:key2}
	Let $M=\begin{pmatrix}
	m_{1,1}& m_{1,2} \\ 
	m_{1,2}& m_{2,2}
	\end{pmatrix} \in \MM_{2}(\R)$ be a symmetric positive definite matrix. For any $b,c\in \R$, there holds
	\[
\lt( (m_{2,2}-c)^2+\frac{(b(m_{2,2}-c)-c(m_{1,2}-b))^2}{\det(M)}  \rt)	(m_{2,2})^{-1}\geq  (m_{1,2}-b)^2	(m_{1,1})^{-1},
	\]
with equality if and only if
\[
m_{2,2}=c+\frac{(m_{1,2}-b)m_{1,2}}{m_{1,1}}.
\]
\end{lemma}

\begin{proof}
	First observe that since $M$ is symmetric positive definite it holds $\det(M)>0$ and	$m_{2,2}>\frac{m_{1,2}^2}{m_{1,1}}$. Then the claimed inequality is equivalent to
	\[
	(m_{2,2}-c)^2\det(M)+(b(m_{2,2}-c)-c(m_{1,2}-b))^2- (m_{1,2}-b)^2	\frac{m_{2,2}}{m_{1,1}}\det(M)\geq 0.
	\]
	Therefore the claimed inequality will follow from the sign study on $\lt(\frac{m_{1,2}^2}{m_{1,1}},+\infty\rt)$ the function
	
	\begin{align*}
g(x)&=(x-c)^2(xm_{1,1}-m_{1,2}^2)+(bx-cm_{1,2})^2- (m_{1,2}-b)^2	\frac{x}{m_{1,1}}(xm_{1,1}-m_{1,2}^2)\\
&=(x^2-2xc+c^2)(xm_{1,1}-m_{1,2}^2)+(b^2x^2-2bcxm_{1,2}+c^2m_{1,2}^2)-(m_{1,2}-b)^2(x^2-x\frac{m_{1,2}^2}{m_{1,1}})\\
&=m_{1,1}x^3+\lt(  -2cm_{1,1}-m_{1,2}^2+b^2 -(m_{1,2}-b)^2 \rt)x^2\\
&+\lt( c^2m_{1,1}+2cm_{1,2}^2-2bcm_{1,2}+ \frac{m_{1,2}^2}{m_{1,1}}(m_{1,2}-b)^2\rt) x -c^2 m_{1,2}^2+c^2m_{1,2}^2\\
&=m_{1,1}x^3-2\lt(  cm_{1,1}+(m_{1,2}-b)m_{1,2} \rt)x^2+\lt( c^2m_{1,1}+2cm_{1,2}(m_{1,2}-b)+ \frac{m_{1,2}^2}{m_{1,1}}(m_{1,2}-b)^2\rt)
x \\
&=x\lt( m_{1,1}x^2-2\lt(  cm_{1,1}+(m_{1,2}-b)m_{1,2} \rt)x+\lt( c^2m_{1,1}+2cm_{1,2}(m_{1,2}-b)+ \frac{m_{1,2}^2}{m_{1,1}}(m_{1,2}-b)^2\rt)
   \rt)=xp(x).
\end{align*}
The discriminant $\Delta$ of the second order polynomial $p$ is given by
\begin{align*}
\Delta&=4\lt(\lt(   cm_{1,1}+(m_{1,2}-b)m_{1,2} \rt)^2-m_{1,1}\lt( c^2m_{1,1}+2cm_{1,2}(m_{1,2}-b)- \frac{m_{1,2}^2}{m_{1,1}}(m_{1,2}-b)^2\rt)\rt)\\
&=4\lt(\lt(   cm_{1,1}+(m_{1,2}-b)m_{1,2} \rt)^2-\lt( c^2m_{1,1}^2+2cm_{1,2}m_{1,1}(m_{1,2}-b)+ (m_{1,2}-b)^2m_{1,2}^2\rt)\rt).\\
&=0
\end{align*}
Therefore 
\[
g(x)=m_{1,1}x\lt(x-\frac{cm_{1,1}+(m_{1,2}-b)m_{1,2}}{m_{1,1}}\rt)^2,
\]
and $g$ is non negative on $\lt(\frac{m_{1,2}^2}{m_{1,1}},+\infty\rt)$  which concludes the proof.

\end{proof}

\begin{lemma}
	\label{lem:novoid}
	There exist $R>0$, $\psi\in \CC^2(\R)$ compactly supported on $[-R,R]$ and $\alpha>0$ satisfying condition \eqref{eq:cond}.
\end{lemma}

\begin{proof}
	
	We set $R=1$ and 
	\[
	\Psi(x)=(x-1)^4(x+1)^4\mb_{x\in [-1,1]},
	\]
	which is $\CC^2$. We numerically check that 
	\[
		\|\Psi\|_{L^\infty}=1, \ \|\Psi'\|_{L^\infty}\simeq 1.9<2 , \ \|\Psi''\|_{L^\infty}=8.
	\]
	Let $A\in \lt( 2\|\Psi'\|_{L^\infty}  ,\|\Psi''\|_{L^\infty}\rt)$ and set $\Gamma=\|\Psi''\|_{L^\infty}-A$. Let $\phi>0$ and fix
	\[
	\alpha=\phi A, \psi(x)=\phi\Psi(x), \gamma =\phi\Gamma.
	\]
	We impose that $\gamma<\|\psi''\|^2_{L^\infty}$, i.e.
	\[
	\phi> \frac{\|\Psi''\|_{L^\infty}-A}{\|\Psi''\|^2_{L^\infty}}. 
	\]
	In view of \eqref{eq:cond}, we also need to impose that $\phi< \|\Psi''\|_{L^\infty}^{-1} 10^{-1}$. \newline
	Therefore $c_*$ in \eqref{eq:cond} is (brutally) bounded by below by
	\begin{align*}
	c_*> \frac{e^{ -\frac{1}{2} }}{10} e^{-\phi} \lt( 1 \wedge \phi \frac{(A -2\|\Psi'\|_{L^\infty})}{4} \rt)>\frac{e^{ -\frac{1}{2} }}{20} e^{-\phi} \phi ,
	\end{align*}
	since we can always choose $A$ close enough to $\|\Psi''\|$ so that $A -2\|\Psi'\|_{L^\infty}>4$. \\
	On the other hand $b_*$ is bounded by above by
	\begin{align*}
	b_*< (\gamma+\|\psi''\|^2_{L^\infty})\frac{4}{3}<\frac{8}{3}\|\psi''\|^2_{L^\infty}=\frac{8}{3}\phi^2\|\Psi''\|^2_{L^\infty}.
	\end{align*}
	We can now choose $A$ close enough to $\|\Psi''\|$, so that $\phi$ maybe small enough so that $c_*>2b_*$
\end{proof}

%
%

%
%
%
%

\end{document}